\documentclass{amsart}
\usepackage{amsfonts}
\usepackage{amsmath,amssymb}
\usepackage{amsthm}
\usepackage{amscd}
\usepackage{graphics}
\usepackage{graphicx}

\theoremstyle{remark}{

\newtheorem{Ex}{{\rm Example}}

\newtheorem{Prob}{{\rm Problem}}

}
\theoremstyle{plain}
{

\newtheorem{Thm}{Theorem}

}

\begin{document}
\title[Morse-Bott functions on 3-dimensional manifolds of certain classes]{Characterizing $3$-dimensional manifolds represented as connected sums of Lens spaces, $S^2 \times S^1$, and torus bundles over the circle by certain Morse-Bott functions}
\author{Naoki kitazawa}

\address{Institute of Mathematics for Industry, Kyushu University, 744 Motooka, Nishi-ku Fukuoka 819-0395, Japan\\
 TEL (Office): +81-92-802-4402 \\
 FAX (Office): +81-92-802-4405 \\
}
\email{n-kitazawa@imi.kyushu-u.ac.jp, naokikitazawa.formath@gmail.com}
\urladdr{https://naokikitazawa.github.io/NaokiKitazawa.html}
\maketitle
\begin{abstract}
We characterize $3$-dimensional manifolds represented as connected sums of Lens spaces, copies of $S^2 \times S^1$, and torus bundles over the circle by certain Morse-Bott functions.
This adds to our previous result around 2024, classifying Morse functions whose preimages containing no singular points are disjoint unions of spheres and tori on closed manifolds represented as connected sums of Lens spaces and copies of $S^2 \times S^1$: we have strengthened and explicitized Saeki's result, characterizing the manifolds via such functions, in 2006. We apply similar arguments. However we discuss in a self-contained way essentially.

\end{abstract}
\section{Introduction.}
\label{sec:1}

{\it Morse}({\it -Bott}) functions have been fundamental and important tools in algebraic topology, differential topology and various geometry of manifolds since the birth of the theory, in the 20th century. 

\cite{milnor} explains related theory for example and we refer to this implicitly. Related to this, fundamental notions, terminologies and notation on smooth manifolds, smooth maps and Morse functions will be reviewed later: we also except readers have related fundamental knowledge and have experienced related studies.

One of natural problems is characterizations of certain classes of manifolds by the existence of certain Morse(-Bott) functions. 

Reeb's theorem \cite{reeb} characterizes spheres topologically except the $4$-dimensional case: we use Morse functions with exactly two {\it singular} points on closed manifolds. The $4$-dimensional unit sphere is characterized by such functions.
Morse functions are, in short, functions such that each singular point is represented by a quadratic form for suitable local coordinates. Singular points of Morse functions appear discretely and have information on homology groups and some information on homotopy of the manifolds. Related to this, more geometrically, they also have information on so-called {\it handle decompositions} of the manifolds, decompositions into disks, shortly. Morse functions are generalized to Morse-Bott functions and at singular points of them, they are represented as the compositions of projections with Morse functions.

Surprisingly, characterizations of certain classes of manifolds via Morse-Bott functions and classifications of these functions on manifolds of certain classes are of still important and difficult studies. A recent study of Saeki \cite{saeki1}, a study on Morse functions such that preimages of single values containing no singular points are disjoint unions of spheres and tori, mainly motivates us.

Hereafter, we use the notation ${\mathbb{R}}^k$ for the $k$-dimensional Euclidean space and $\mathbb{R}:={\mathbb{R}}^1$. This is also a Riemannian manifold with the standard Euclidean metric and we define $||x|| \geq 0$ as the distance of $x$ and the origin $0$ there. 
Let
$S^k:=\{x \in {\mathbb{R}}^{k+1} \mid ||x||=1\}$ and $D^k:=\{x \in {\mathbb{R}}^{k} \mid ||x|| \leq 1\}$. They are the $k$-dimensional unit sphere and the $k$-dimensional unit disk. {\it Lens spaces} and {\it torus bundles} over the circle $S^1$ are of important $3$-dimensional closed, connected and orientable manifolds. We explain this later again.
\begin{Thm}[{\cite[Theorem 6.5]{saeki1}}]

\label{thm:1}

	A $3$-dimensional closed, connected and orientable manifold $M$ admits a Morse function $f:M \rightarrow \mathbb{R}$ such that preimages of single values containing no singular points are disjoint unions of copies of the sphere $S^2$ and the torus $S^1 \times S^1$ if and only if
$M$ is diffeomorphic to $S^3$, $S^1 \times S^2$, a Lens space, or a manifold represented as a connected sum of these manifolds.
\end{Thm}
The following is our main result, extending Theorem \ref{thm:1} to a certain class of Morse-Bott functions.
\begin{Thm}[Our main result]
\label{thm:2}
	A $3$-dimensional closed, connected and orientable manifold $M$ admits a Morse-Bott function $f:M \rightarrow \mathbb{R}$ such that the following hold if and only if
$M$ is diffeomorphic to $S^3$, $S^1 \times S^2$, a lens space, a torus bundle over the circle $S^1$ or a manifold represented as a connected sum of these manifolds.
\begin{enumerate}

\item \label{thm:2.1} Preimages of single values of $f$ containing no singular points are disjoint unions of copies of the sphere $S^2$ and the torus $S^1 \times S^1$.
\item \label{thm:2.2} The set of all singular points of the Morse-Bott function $f$ is a disjoint union of {\rm (}copies of{\rm )} the following four manifolds.
\begin{enumerate}
	\item \label{thm:2.2.1}
 A one-point set.
	\item \label{thm:2.2.2} The circle $S^1$.
	\item \label{thm:2.2.3} The torus $S^1 \times S^1$ or the sphere $S^2$.
	\item \label{thm:2.2.4} The projective plane ${{\mathbb{R}}P}^2$.
\end{enumerate}
\item \label{thm:2.3} Around each singular point of $f$ where $f$ has no local extremum, it is a Morse function.
\end{enumerate}
\end{Thm}

Note that this gives an answer to \cite[Problem 1]{kitazawa4}. 

Related to this, Morse-Bott functions on closed surfaces and their {\it Reeb-digraphs}, which are {\it digraphs} or oriented graphs, are recently studied actively, surprisingly (\cite{gelbukh1, gelbukh2, gelbukh3, gelbukh4, gelbukh5, marzantowiczmichalak, michalak1, michalak2}). The {\it Reeb digraph} of a smooth function is the space of all connected components of preimages of single values and defined as a quotient space of the manifold, its vertex is defined as a point (component) containing singular points of the function, and the orientations (of edges) are defined according to the values of the functions naturally. 
\cite{kitazawa4} is another related study and studies Morse functions in Theorem \ref{thm:2}. We respect \cite{kitazawa4} here. \cite{michalak3} also contains a related $3$-dimensional study.

The next section is for preliminaries. The third section is for the proof and related arguments and remarks. Note also that \cite{kitazawa4} is closely related and that we respect and refer to this in some scenes. However we do not assume essential arguments of the preprint \cite{kitazawa4} unless we need related arguments. We believe that our new present result is sufficiently important to be presented in a single form. We also present arguments of certain new types in our paper. 

\section{Preliminaries.}

Topological manifolds are known to be naturally regarded as so-called CW complexes. Smooth manifolds are, more strongly, regarded as polyhedra. Graphs are $1$-dimensional polyhedra (if they have edges).
For a non-empty space $X$ of such classes, we can define the {\it dimension} $\dim X \geq 0$ as an integer uniquely.

For a manifold $X$, we can define the interior ${\rm Int}\ X$, a manifold of dimension $\dim X$ and the boundary $\partial X:=X-{\rm Int}\ X$, a empty set or a manifold of dimension $\dim X-1 \geq 0$.

Let $c:X \rightarrow Y$ be a smooth map between smooth manifolds $X$ and $Y$. A point $p \in X$ is a singular point of $c$ if the rank of the differential is smaller than both the dimensions $\dim X$ and $\dim Y$ there. A Morse function $c:X \rightarrow \mathbb{R}$ is a 
smooth function such that its singular point is always in the interior ${\rm Int}\ X$ and that each singular point $p$ of the function is represented by the form $c(x_1,\cdots x_m)={\Sigma}_{j=1}^{m-i(p)} {x_j}^2-{\Sigma}_{j=1}^{i(p)} {x_{m-i(p)+j}}^2+c(p)$ for some local coordinates and some integer $0 \leq i(p) \leq m$. Here $i(p)$ is shown to be uniquely defined and we call it the {\it index} of $p$ {\it for $c$}. We do not need such notions here.
Singular points of a Morse function appear discretely. It is a fundamental fact from singularity theory that we can also deform a Morse function via a small homotopy into one such that at distinct singular points of it the values are distinct. We call such a Morse function a {\it simple} Morse function. \cite{golubitskyguillemin} explains singularity theory of differentiable functions and maps systematically from elementary terminologies, notions and arguments, for example.

Hereafter, we call a Morse function such that preimages of single values containing no singular points of it are disjoint unions of copies of $S^2$ and $S^1 \times S^1$ as studied in Theorem \ref{thm:1} a {\it sphere-torus-fibered} Morse function or an {\it STF} Morse function. The name respects \cite{kitazawa4}.

A {\it Morse-Bott} function is a smooth function which is at each singular point represented as the composition of a smooth map with no singular points with a Morse function for suitable local coordinates.

Graphs are polyhedra consisting of 1-cells ({\it edges}) and 0-cells ({\it vertices}). The {\it edge set} of a graph is the set of all edges of it. The {\it vertex set} of the graph is the set of all vertices of it. We can orient edges in an arbitrary way and the graph is an {\it oriented graph} or a {\it digraph}. Two graphs are said to be {\it isomorphic} if there exists a piecewise smooth homeomorphism between them mapping the vertex set of one graph onto that of another graph. This map is called an {\it isomorphism} between these graphs.
Two digraphs are said to be {\it isomorphic} if they are isomorphic as graphs and admit an isomorphism of the graphs mapping each edge of one graph into an edge of the other graph preserving the orientations. The {\it degree} of a vertex $v$ is defined as the number of edges containing the vertex $v$.
 
The {\it Reeb space} $W_c$ of a map $c:X \rightarrow Y$ between topological spaces is defined as the quotient space $X/{\sim}_c$ by the equivalence relation: $x_1 {\sim}_c x_2$ if and only if $x_1$ and $x_2$ belong to a connected component of a same preimage $c^{-1}(y)$. We have the quotient map $q_c:X \rightarrow W_c$ and in a unique way a map $\bar{c}:W_c \rightarrow \mathbb{R}$ with the relation $c=\bar{c} \circ q_c$. 

Let $Y:=\mathbb{R}$. By defining the vertex set of $W_c$ as the set of all connected components containing some singular points of $c:X \rightarrow \mathbb{R}$ with the condition that $\partial X$ is empty, this is a graph in considerable cases. We call the graph the {\it Reeb graph} of $c$. More rigorously, \cite{saeki2, saeki3} discuss this rigorously. Our cases are of such cases unless otherwise stated. Reeb graphs are also classical objects and have been fundamental and strong tools in understanding the manifolds compactly (\cite{reeb}).

The Reeb graph $W_c$ has the structure of a digraph in a canonical way. We orient an edge connecting $v_1$ and $v_2$ as an edge departing from $v_1$ and entering $v_2$ if $\bar{c}(v_1)<\bar{c}(v_2)$. We can easily check that the function $\bar{c}$ is regarded as a piecewise smooth function and injective on each edge. We call this digraph $W_c$ the {\it Reeb digraph} of $c$. If there exists a piecewise smooth function $g:K \rightarrow \mathbb{R}$ on a graph $K$ such that on each edge it is injective, then we can naturally have a digraph $K_g$.

The set of all smooth maps from a smooth manifold $X$ into another smooth manifold $Y$ is topologized with the so-called {\it $C^{\infty}$ Whitney topology}: in short this respects derivatives of the maps. See \cite{golubitskyguillemin} again. We do not need to understand precise and rigorous definitions and arguments on this here.
A {\it diffeomorphism} is a smooth map which is a homeomorphism and which has no singular point. A {\it diffeomorphism on a smooth manifold} $X$ is a diffeomorphism from $X$ onto $X$. The {\it diffeomorphism group} of $X$ is the group of all diffeomorphisms on $X$.

In our paper, a bundle is a smooth bundle, or a bundle whose fiber is a smooth manifold and whose structure group is a subgroup of the diffeomorphism group of the smooth manifold.  

A {\it circle bundle} (resp. {\it torus bundle}) is a smooth bundle whose fiber is the circle $S^1$ (resp. torus $S^1 \times S^1$). We also assume that the diffeomorphisms do not reverse the orientations.

For $3$-dimensional manifold theory, see \cite{hempel}. The class of lens spaces contains circle bundles over $S^2$. These circle bundles are distinguished by the $1$-st integral homology groups, isomorphic to finite cyclic groups. Lens spaces may not be circle bundles. However, the $1$st integral homology group and the fundamental group of each Len space are isomorphic and of such groups. 
Torus bundles over the circle are not diffeomorphic to manifolds represented as connected sums of Lens spaces and copies of $S^1 \times S^2$.

\section{Our main result, Theorem \ref{thm:2}, and related arguments and remarks.}
\begin{proof}[A proof of Theorem \ref{thm:2}]
We refer to several figures from \cite{kitazawa4} where we do not assume essential arguments first presented in \cite{kitazawa4}. 
Some arguments here are presented in \cite{kitazawa4} and several existing studies. \\
\ \\
STEP 1 The manifold $M$ of the domain of the Morse-Bott function $f:M \rightarrow \mathbb{R}$ must be of the presented class of $3$-dimensional manifolds. \\

\ \\
STEP 1-1  The preimage ${q_f}^{-1}(N(v))$ of a small regular neighborhood $N(v)$ of a vertex $v$ where $f$ has a local extremum. \\
We investigate the preimage ${q_f}^{-1}(N(v))$ of a small regular neighborhood $N(v)$ of a vertex $v$ where $f$ has a local extremum. We respect the four cases from (\ref{thm:2.2}). \\
\ \\
CASE 1-1-1 \\
In the first case (\ref{thm:2.2.1}), the preimage is diffeomorphic to the disk $D^3$ and we have a Morse function having exactly one singular point in the interior. It is regarded as a so-called height function of the disk $D^3$. \\
\ \\
CASE 1-1-2 \\
In the second case (\ref{thm:2.2.2}), the preimage is diffeomorphic to the product $S^1 \times D^2$ and we have a Morse-Bott function the set of all singular points of which is seen as the set $S^1 \times \{0\} \subset S^1 \times {\rm Int}\ D^2$. This can be deformed into a simple STF Morse function with exactly two singular points by a suitable small homotopy, thanks to fundamental singularity theory and elementary topological arguments. \\
\ \\
CASE 1-1-3 \\
In the third case (\ref{thm:2.2.3}), the preimage is diffeomorphic to the product $S^1 \times S^1 \times D^1$ or $S^2 \times D^1$ and we have a Morse-Bott function the set of all singular points of which is seen as the set $S^1 \times S^1 \times \{0\} \subset S^1 \times S^1 \times {\rm Int}\ D^1$ or $S^2 \times \{0\} \subset S^2 \times {\rm Int}\ D^1$. The latter case of these two functions can be deformed into a simple STF Morse function with exactly two singular points by a suitable small homotopy, thanks to fundamental singularity theory and elementary topological arguments.

Note that the former function here is closely related to a key ingredient in STEP 1-2 and our new arguments. \\
\ \\
CASE 1-1-4 \\
In the fourth case (\ref{thm:2.2.4}), the preimage is shown to be diffeomorphic to the complementary set of the interior of a smoothly embedded disk $D^3$ in the $3$-dimensional real projective space ${\mathbb{R}P}^3$ by a fundamental argument on $3$-dimensional manifolds. Note that the $3$-dimensional real projective space is a circle bundle over $S^2$, for example. We have a Morse-Bott function the set of all singular points of which is seen as a copy of the real projective plane ${\mathbb{R}P}^2$ embedded in a certain canonical way in the $3$-dimensional real projective space ${\mathbb{R}P}^3$. \\
\ \\
STEP 1-2 Deforming the Morse-Bott function $f$ locally around the preimages ${q_f}^{-1}(N(v))$ of small regular neighborhoods $N(v)$ of vertices $v$ where $f$ has no local extremum. \\ 
 \ \\
By a suitable small homotopy, we can deform the Morse-Bott function around the preimages ${q_f}^{-1}(N(v))$ of small regular neighborhoods $N(v)$ of vertices $v$ where $f$ has no local extremum preserving the function on the complementary set of the interior of the disjoint union of these local preimages ${q_f}^{-1}(N(v))$. The local functions are originally Morse by (\ref{thm:2.3}), and by (\ref{thm:2.1}), STF Morse functions. By \cite[Lemma 6.6]{saeki1}, these local functions are changed into simple STF Morse functions.
Such an argument is also presented in \cite{kitazawa4} and we present again.
 
We also present additional arguments and figures from \cite{kitazawa4}. 
FIGURE \ref{fig:1} shows local forms of simple STF Morse functions by Reeb digraphs with information on preimages. 

Blue (red) colored edges show edges the preimages of single points in the interior of which are diffeomorphic to $S^2$ (resp. $S^1 \times S^1$). Black dots are for vertices. We adopt this rule for our figures. We can consider the case of $-c$ for the function $c$ and we omit the case. We also respect this fundamental rule, hereafter.
\begin{figure}
		\includegraphics[width=80mm,height=27.5mm]{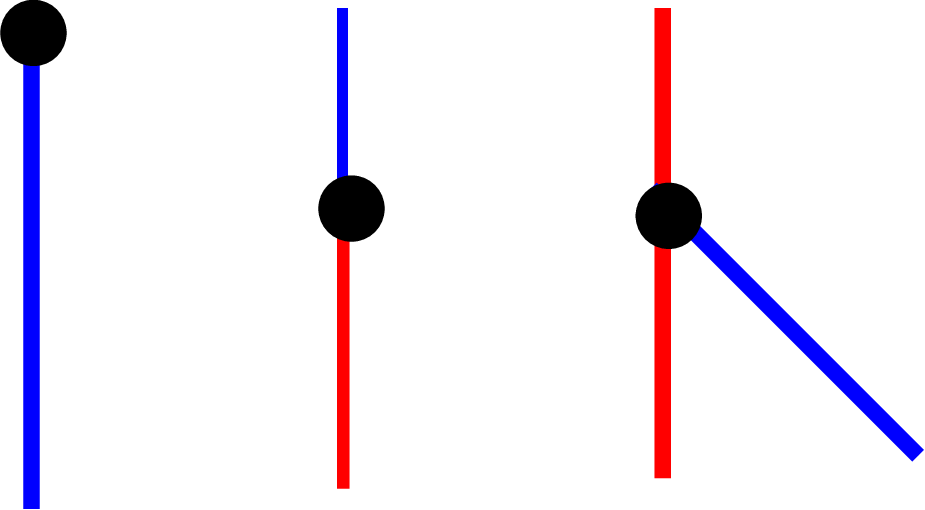}
		\caption{Local information on the Reeb (di)graph and preimages for simple STF Morse functions. Blue (red) colored edges show edges the preimages of single points in the interior of which are diffeomorphic to $S^2$ (resp. $S^1 \times S^1$). Black dots are for vertices. Hereafter we respect this rule for our figures.}
\label{fig:1}
	\end{figure}

FIGURE \ref{fig:2} shows local changes. Each homotopy is realized by a suitable small homotopy. This is from fundamental singularity theory together with the fundamental fact on the 3-dimensional manifold theory: the preimage of the presented local graph (local complex) is diffeomorphic to a manifold obtained by removing the interiors of two disjointly and smoothly embedded copies of the disk $D^3$ in the interior of $S^1 \times D^2$. 
\begin{figure}
		\includegraphics[width=80mm,height=27.5mm]{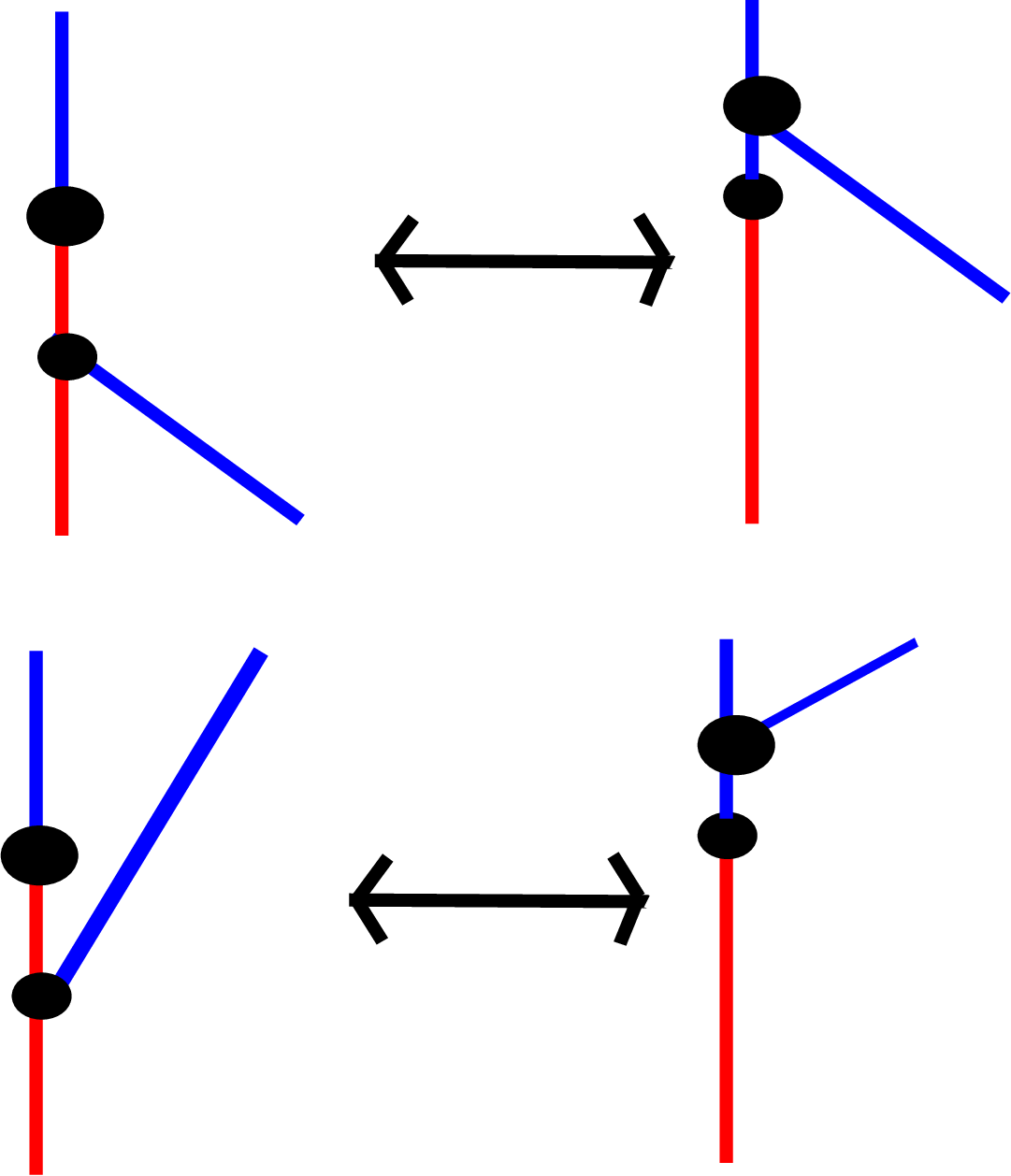}
		\caption{Local changes of (the Reeb digraphs of) simple STF Morse functions.}
\label{fig:2}
	\end{figure}

We can consider a suitable iteration
 of local changes of functions presented in FIGURE \ref{fig:2} to have another new Morse-Bott function $f_0:M \rightarrow \mathbb{R}$ and a new Reeb digraph such that the union of all edges the preimages of single points in the interiors of which are diffeomorphic to $S^1 \times S^1$ consists of connected components represented as either of the following. If we need, then we change the local functions for CASE 1-1-2 into simple STF Morse functions locally.

Note that for Reeb spaces and Reeb graphs, preimages are considered for the map $q_c:W_c \rightarrow Y$ for a given map $c:X \rightarrow Y$. We also remember the map $\bar{c}$ with the relation $c=\bar{c} \circ q_c$.

We note that, hereafter, we present some new arguments, which do not appear in \cite{kitazawa4}. 

\begin{itemize}
\item A closed arc $I$ satisfying the following: vertices of the boundary points, points of $\partial I$, are of degree $2$ and ones where the function $\bar{f_0}:W_{f_0} \rightarrow \mathbb{R}$ does not have a local maximum, and vertices in the interior are of degree $2$ and ones where the function $\bar{f_0}:W_{f_0} \rightarrow \mathbb{R}$ has a local extremum. We can also check that the preimage ${q_f}^{-1}(N(I))$ of a small regular neighborhood $N(I)$ of $I$ is diffeomorphic to a manifold obtained by removing the interiors of two smoothly and disjointly embedded copies of $D^3$ from a Lens space, $S^2 \times S^1$, or $S^3$.
\item A connected component $C$ homeomorphic to the circle $S^1$ satisfying the following: at each vertex $v \in C$ there where the restriction ${\bar{f_0}} {\mid}_{C}$ to $C$ of the function $\bar{f_0}:W_{f_0} \rightarrow \mathbb{R}$ has a local extremum, the original function $\bar{f_0}$ also has a local extremum at the vertex $v$ and $v$ is also of degree $2$ in the graph $W_{f_0}$. We can also check that the preimage ${q_f}^{-1}(N(C))$ of a small regular neighborhood $N(C)$ of $C$ is diffeomorphic to a manifold obtained by removing the interiors of finitely many smoothly and disjointly embedded copies of $D^3$ from a torus bundle over the circle.
\end{itemize}

Remembering STEP 1-1 with CASE 1-1-1--CASE 1-1-4 and Theorem \ref{thm:1} for example, we can see that the closed, connected and orientable manifold $M$ of the domain is diffeomorphic to a manifold obtained in the following way.
\begin{itemize}
\item We prepare a suitable finite family of torus bundles over $S^1$, Lens spaces, copies of $S^2 \times S^1$, and copies of $S^3$. 
\item We remove the interiors of finitely many smoothly and disjointly embedded copies of $D^3$ from the prepared manifolds suitably.
\item We glue the resulting manifolds suitably to have our desired closed, connected and orientable manifold along the boundaries, each of which is diffeomorphic to $S^2$. The resulting manifold is diffeomorphic to $M$.
\end{itemize}
Our manifold $M$ is also of the class of $3$-dimensional closed, connected and orientable manifolds presented in the statement of Theorem \ref{thm:2}. \\

This completes STEP 1. \\
\ \\
STEP 2 Construct a desired Morse-Bott function on a $3$-dimensional manifold $M$ of the presented class, conversely. \\

On a sphere $S^3$, we have a Morse function with exactly two singular points. It is also a specific case of simple STF Morse functions.

We can reconstruct a desired simple STF Morse function on each lens space $M_j$, $S^1 \times S^2$  and $S^3$ from the colored digraph in FIGURE \ref{fig:3}. Rules for colors are for preimages as presented: blue (red) colored edges are for preimages of single points diffeomorphic to $S^2$ (resp. $S^1 \times S^1$). Blue colored vertices are for removal of (the interiors of) small regular neighborhoods of the vertices and the preimages: preimages here are diffeomorphic to $D^3$ (${\rm Int}\ D^3$). 

In addition, we can also reconstruct a desired Morse-Bott function $f_j:M_j \rightarrow \mathbb{R}$ on each torus bundle $M_j$ over the circle $S^1$ from the colored digraph  in FIGURE \ref{fig:4}. Rules for colors and blue vertices are same as those in the previous situation.
\begin{figure}
		\includegraphics[width=80mm,height=27.5mm]{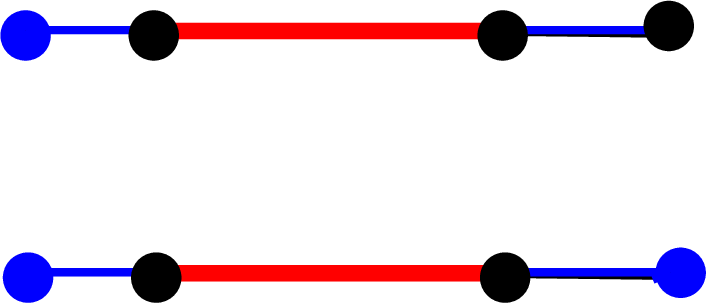}
		\caption{We can reconstruct a desired Morse function on each lens space $M_j$, $S^1 \times S^2$ and $S^3$, from the colored digraph. Rules for colors are for preimages as presented. Blue colored vertices are for removal of (the interiors of) small regular neighborhoods of the vertices and the preimages, each of which is diffeomorphic to $D^3$ (${\rm Int}\ D^3$).}
\label{fig:3}
	\end{figure}
\begin{figure}
		\includegraphics[width=80mm,height=27.5mm]{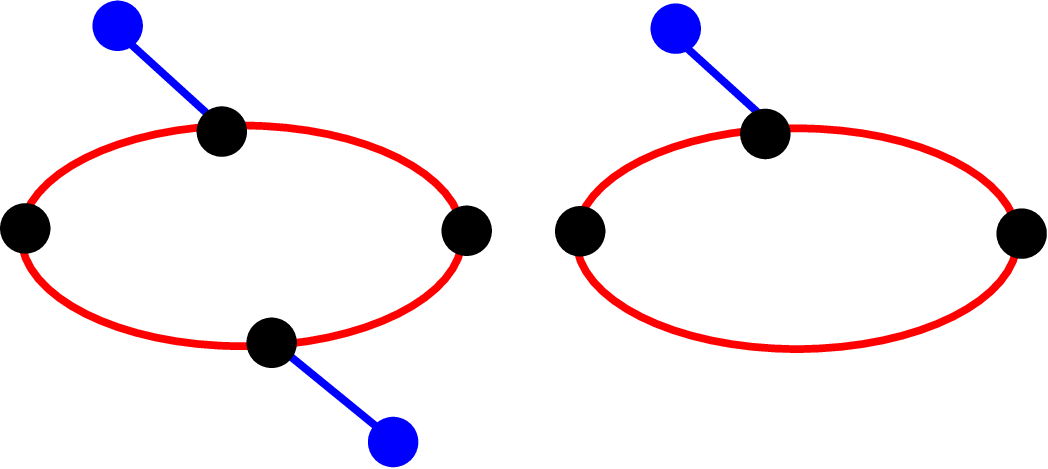}
		\caption{We can reconstruct a desired Morse-Bott function on each torus bundle $M_j$ over the circle $S^1$, from the colored digraph.}
\label{fig:4}
	\end{figure}

On a general $M$, represented as a connected sum of the manifolds $M_j$, by gluing these local Morse-Bott functions suitably, we have a desired Morse-Bott function $f:M \rightarrow \mathbb{R}$. 

For our reconstruction here, the author has discussed more general cases in \cite{kitazawa1, kitazawa2} for example. However we do not understand these studies of course. Our construction in the present proof is very fundamental, natural and specific. 

Note also that we can have a resulting Reeb digraph as a so-called {\it cactus graph}: a {\it cactus graph} is a graph such that each edge is contained in at most one circle ({\it simple cycle}).  
\\
\ \\
This completes the proof.
\end{proof}
\begin{Prob}
\label{prob:1}
	For a piecewise smooth function $g:K \rightarrow \mathbb{R}$ on a finite and connected graph $K$ such that on each edge $e$ it is injective, consider the digraph $K_g$ such that at each vertex $v$ where $g$ has a local extremum the degree of $v$ is $1$ or $2$. For the edge set $E_{K_g}$, a map $l_{K_g}$ is defined satisfying the following: the values are $0$ or $1$ and on edges $e_v$ containing a vertex $v$ where the function $g$ has a local extremum, the values $l(e_v)$ are constant (for each $v$). 
	Does there exist a Morse-Bott function $f:M \rightarrow \mathbb{R}$ on a given $3$-dimensional manifold $M$ satisfying the three conditions in Theorem \ref{thm:2} and the following properties?
	\begin{enumerate}
		\item There exists an isomorphism $\phi:K_g \rightarrow W_f$ between the digraphs.
		\item The preimage ${q_f}^{-1}(p)$ of a point $p$ in the interior of the edge $\phi(e) \subset W_f$ is diffeomorphic to $S^2$ {\rm (}$S^1 \times S^1${\rm )} if $l_{K_g}(e)=0$ {\rm (}resp. $1${\rm )}.
	\end{enumerate} 
\end{Prob}

This respects a problem mainly discussed in \cite{kitazawa4} (\cite[Problem 1 ]{kitazawa4}).
For closed surfaces, \cite{michalak1} is a complete answer for the Morse function case and \cite{gelbukh4, gelbukh5} give some complete answers for Morse-Bott functions on closed surfaces.

We present a kind of explicit affirmative answers characteristic to our situation.
A {\it subgraph} $K^{\prime}$ of a graph $K$ means a subcomplex of $K$. If the graph $K$ is a digraph, then $K^{\prime}$ is also a digraph canonically.
\begin{Thm}
\label{thm:3}
Given a situation for Problem \ref{prob:1} as follows.

Let there exist $l$ disjoint simple cycles $\{C_j\}_{j=1}^l$ in $K$ satisfying the following. We consider the restriction $g {\mid}_{C_j}$ of the function $g$. The following are satisfied.
\begin{itemize}
\item For each edge $e_{C_j}$ of $C_j$, $l_{K_g}(e_{C_J})=1$.
\item At each vertex where $g {\mid}_{C_j}$ has a local extremum, it is also a vertex where the original function $g$ has a local extremum.
\item If $K$ is homeomorphic to the circle $S^1$, then at some vertex $v \in K$, $g$ does not have a local extremum. 
\end{itemize}
Then for any family $\{M_j\}_{j=1}^l$ of $l$ manifolds each of which is a torus bundle over $S^1$ and "some $3$-dimensional manifold $M$ from Theorem \ref{thm:1}", denoted by $M_{l+1}$, and a manifold $M$ represented as a connected sum of these $l+1$ manifolds, we can have a desired Morse-Bott function $f:M \rightarrow \mathbb{R}$.
	
\end{Thm}
\begin{proof}
Here, we refer to several published articles \cite{kitazawa1, kitazawa2} and arguments. We also refer to some preprints of the author \cite{kitazawa3, kitazawa4}. 

We can have the following subgraph $K^{\prime}$ of $K$ in the unique way.
\begin{itemize}
\item The subgraph $K^{\prime}$ contains no edge in cycles $C_j$. 
\item The subgraph $K^{\prime}$ is maximal among all subgraphs $K^{\prime \prime}$ satisfying the previous condition.
\end{itemize} 
We do the following operation for each vertex $v_{C}$ of $K^{\prime}$ originally contained in some cycle $C_j$. 

We add two oriented edges $e_{v_{C,K^{\prime}},{\rm l}}$ and $e_{v_{C,K^{\prime}},{\rm h}}$. We also add the first edge $e_{v_{C,K^{\prime}},{\rm l}}$ as an edge departing from a new vertex $v_{C,K^{\prime},{\rm l}}$ and entering $v_{C}$. We also add the second edge $e_{v_{C,K^{\prime}},{\rm h}}$ as an edge departing from the vertex $v_{C}$ and entering a new vertex $v_{C,K^{\prime},{\rm h}}$.

Let ${K_0}^{\prime}$ denote the resulting digraph. We extend the map $l_{K_g}$ as a map such that at each new edge the value is $0$. Let the map denoted by ${l_{K_g,0}}^{\prime}$.

We also do the following operation for each vertex $v_{C}$ of each cycle $C_j$. The notation "$v_{C}$ here" is same as the vertex $v_{C}$ before.

We add another two oriented edges $e_{v_{C},{\rm l}}$ and $e_{v_{C},{\rm h}}$. We also add the first edge $e_{v_{C},{\rm l}}$ as an edge departing from a new vertex $v_{C,{\rm l}}$ and entering $v_{C}$. We also add the second edge $e_{v_{C},{\rm h}}$ as an edge departing from the vertex $v_{C}$ and entering a new vertex $v_{C,{\rm h}}$. Let $C_{0,j}$ denote the resulting digraph, obtained as a result of the change from $C_j$. 

We further extend the maps $l_{K_g}$ and ${l_{K_g,0}}^{\prime}$ as a map such that at each of additional new edges the value is $0$. Let the map denoted by $l_{K_g,0}$.

For ${K_0}^{\prime}$, we can have a Morse-Bott function $f^{\prime}:M^{\prime} \rightarrow \mathbb{R}$ on a suitably chosen manifold $M^{\prime}$ as in the statement of Theorem \ref{thm:2}: for the preimages we respect the map ${l_{K_g,0}}^{\prime}$ (and the map $l_{K_g,0}$). 
We can construct the function in such a way that the set of all singular points of the function is a disjoint union of copies of $S^2$, $S^1$, $S^1 \times S^1$ and one-point sets: we consider CASE 1-1-1, CASE 1-1-2 and CASE 1-1-3 in the proof of Theorem \ref{thm:2}.
For the class of manifolds $M^{\prime}$ belongs to, we need additional consideration on local construction around each vertex $v$ where for the digraph ${K_0}^{\prime}$ both an edge entering $v$ and an edge departing from $v$ exist. We construct locally as in \cite[(3) in STEP 2 in the proof of Theorem 2]{kitazawa4}. Or we also construct as presented in the articles \cite{kitazawa1, kitazawa2}: for example \cite[Proof of Theorem 1.2]{kitazawa1}. In this construction, it is most important to choose the local function so that we can deform the local function by a suitable homotopy into a simple STF Morse function having the preimage of some single value diffeomorphic to $S^2$ and containing no singular point of the simple STF Morse function.

Our construction also yields a Morse-Bott function ${f_0}^{\prime}:M^{\prime} \rightarrow \mathbb{R}$ on $M^{\prime}$ which is, around each point where the function does not have a local extremum, represented as a simple STF Morse function. In addition, the function ${f_0}^{\prime}$ satisfies the following. The union of all edges the preimages of single points in the interiors of which are diffeomorphic to the torus $S^1 \times S^1$ is a disjoint union of intervals: as a small remark, we deform the local function around a vertex for CASE 1-2 in the proof of Theorem \ref{thm:2} into a local simple STF Morse function, if we need. For this, note that for any simple cycle ${C_{{K_0}^{\prime}}}^{\prime}$ in the graph ${K_0}^{\prime}$, at some vertex $v \in {C_{{K_0}^{\prime}}}^{\prime}$, either $g {\mid}_{{C_{{K_0}^{\prime}}}^{\prime}}$ does not have a local extremum or $g$ does not have a local extremum on the original graph $K$. This is thanks to the assumption that if the graph $K$ is homeomorphic to the circle $S^1$, then at some vertex $v \in K$, $g$ does not have a local extremum. Each connected component of the manifold $M^{\prime}$ belongs to the class of "the $3$-dimensional manifolds $M$ of Theorem \ref{thm:1}". Remember also that such an argument is presented in STEP 1-2 in the proof of Theorem \ref{thm:2}. By the structure of the Morse-Bott function and fundamental arguments on Morse functions, we have a copy ${D^3}_{v_{C}}$ of the disk $D^3$ small and smoothly embedded into $M^{\prime}$ and mapped onto the union of the two edges $e_{v_{C,K^{\prime}},{\rm l}}$ and $e_{v_{C,K^{\prime}},{\rm h}}$ for each vertex $v_C$ before. We can also choose the disk ${D^3}_{v_{C}}$ as a space containing exactly two singular points of the function ${f_0}^{\prime}$. For related arguments, for our previous result, see also \cite[A proof of Main Theorem 1]{kitazawa3} for example.

For each simple cycle $C_j$, we remove all vertices $v_{C_j}$ where the function does not have local extrema from $C_j$ and have a new simple cycle ${C_j}^{\prime}$. We can construct a Morse-Bott function on $f_j:M_j \rightarrow \mathbb{R}$ and the Reeb digraph isomorphic to ${C_j}^{\prime}$ and the preimage of a point in $C^{\prime}$ is diffeomorphic to the torus $S^1 \times S^1$: the point may be a point in the interior of an edge or a vertex. We can also reconstruct a Morse-Bott function $f_{0,j}:M_j \rightarrow \mathbb{R}$ which is, around each vertex where the function does not have a local extremum, represented as an STF Morse function not being simple, whose Reeb digraph is isomorphic to $C_{0,j}$ and which respects the map $l_{K_g,0}$ for the preimages. We can reconstruct the functions so that $f_j$ is deformed into $f_{0,j}$ by a suitable homotopy, thanks to fundamental arguments on Morse functions, singular points of them and (generating and canceling) handles. We also have a copy ${D^3}_{v_{C_j}}$ of the disk $D^3$ small and smoothly embedded into $M_j$ and mapped onto a closed interval embedded in the interior of the union of two adjacent edges $e_{v_{C_j},{\rm l}}$ and $e_{v_{C_j},{\rm h}}$ for each vertex $v_{C_j}$ before. Remember that for these edges, the notation of the form "$e_{v_{C},{\rm l}}$, $e_{v_{C},{\rm h}}$, and $v_{C}$" is used before. We have no problem on the usage.
We can also choose the disk ${D^3}_{v_{C_j}}$ as a space containing no singular points of the function $f_{0,j}$. For related arguments, for our previous result, see also \cite[A proof of Main Theorem 1]{kitazawa3} again, for example.

Respecting these arguments and information, we construct a new $3$-dimensional manifold $M$.
First we remove the interiors of disks  ${D^3}_{v_{C}}$ and ${D^3}_{v_{C_j}}$ from $M_j$ and $M^{\prime}$. We glue the resulting manifolds along connected components of the boundaries one after another suitably to have a closed, connected and orientable manifold $M$. More explicitly, we glue connected components originally corresponding to a same vertex $v_{C}=v_{C_j} \in K$ where abusing the notation in such a way has no problem. We also glue the functions $f^{\prime}$ and $f_{0,j}$ (, restricted to the remaining manifolds).

We can have a new Morse-Bott function $f^{\prime}:M \rightarrow \mathbb{R}$ on the new manifold $M$ whose Reeb digraph is isomorphic to a digraph obtained by identifying $e_{v_{C,K^{\prime}},{\rm l}}$ with $e_{v_{C},{\rm l}}$, and $e_{v_{C,K^{\prime}},{\rm h}}$ with $e_{v_{C},{\rm h}}$ from ${K_0}^{\prime}$ and $C_{0,j}$, canonically. By the structure of the Morse-Bott function and fundamental arguments on Morse functions, especially on singular points of the functions, handles and so-called canceling pairs of singular points of the functions or handles, we have a desired function $f:M \rightarrow \mathbb{R}$. Last, we can also see that $M$ is represented as in the statement.

This completes the proof.

\end{proof}
\begin{Ex}

From the two graphs in FIGURE \ref{fig:5}, seen naturally as digraphs, we can reconstruct Morse-Bott functions as in Theorems \ref{thm:2} and \ref{thm:3}.
\begin{figure}
		\includegraphics[width=80mm,height=45mm]{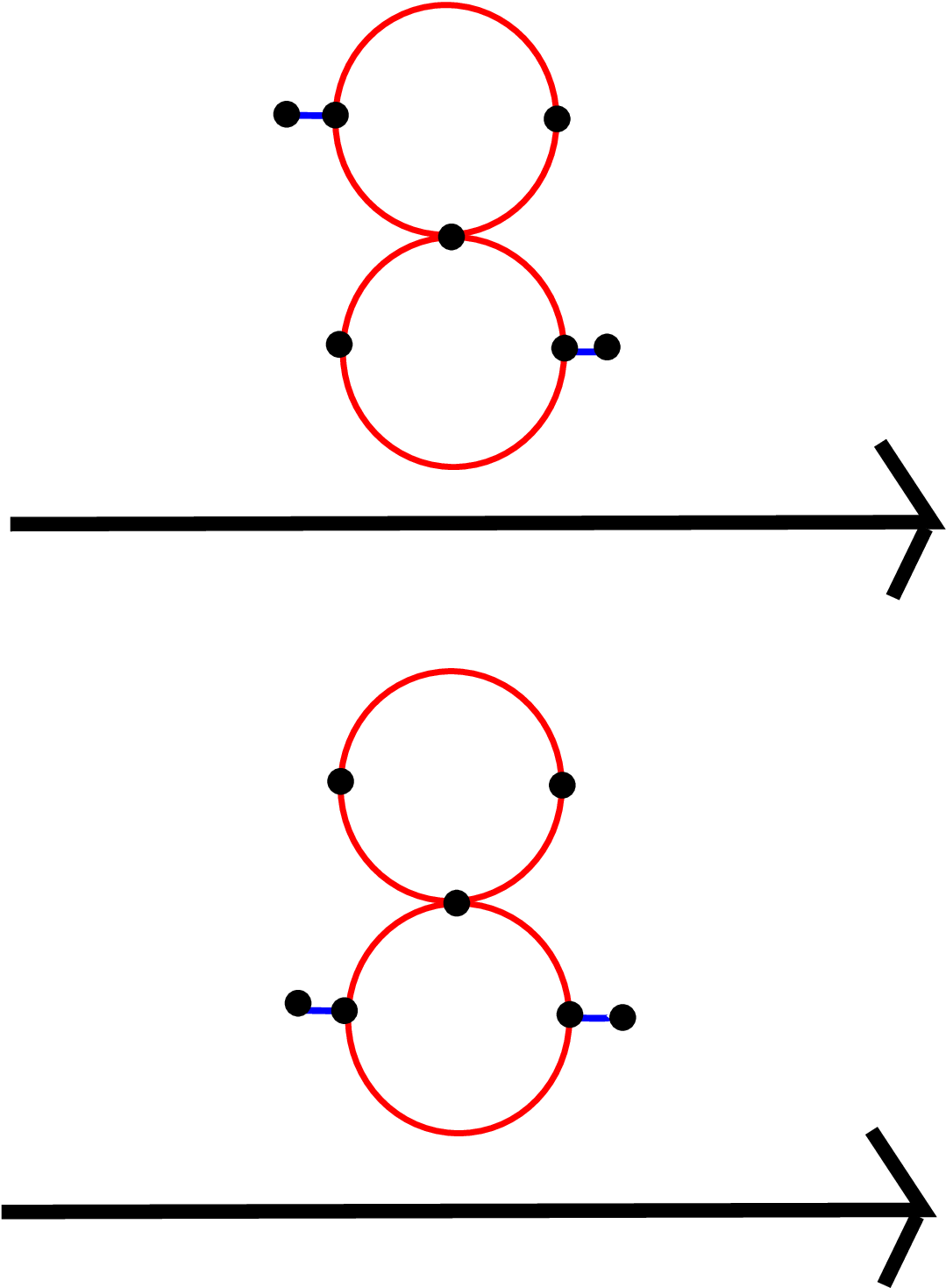}
		\caption{Two similar digraphs for reconstruction in Theorems \ref{thm:2} and \ref{thm:3}.
 For example, rules for colored edges and vertices are as presented before. Long arrows are for orientations of the digraphs. The former (latter) digraph is for $l=0$ (resp. $l=0, 1$) in Theorem \ref{thm:3}.}
\label{fig:5}
	\end{figure}
As before, a blue colored edge $e$ implies $l_K(e)=0$ and a red colored edge $e$ implies $l_K(e)=1$.
 It is important that from the first case, we must have a manifold $M$ as the domain as a "$3$-dimensional manifold $M$ in Theorem \ref{thm:1}" and that from the second case, we can have a case of $l=1$ in Theorem \ref{thm:3}.
\end{Ex}

\section{Acknowledgement.}
The author would like to thank members of the research group hosted by Osamu Saeki. The author would also like to thank people organizing and supporting Saga Souhatsu Mathematical Seminar (http://inasa.ms.saga-u.ac.jp/Japanese/saga-souhatsu.html): the author would like to thank Inasa Nakamura again for inviting the author as a speaker and letting him present \cite{kitazawa1, kitazawa3}. Discussions on several Morse functions and topological properties of the manifolds of their domains with these groups have motivated the author to challenge \cite{kitazawa4} and further and led the author to the present result. 
\section{Conflict of interest and Data availability.}
\noindent {\bf Conflict of interest.} \\
The author works at Institute of Mathematics for Industry (https://www.jgmi.kyushu-u.ac.jp/en/about/young-mentors/). This project is closely related to our study. Our study thanks them for their encouragements. The author is also a researcher at Osaka Central
Advanced Mathematical Institute (OCAMI researcher): this is supported by MEXT Promotion of Distinctive Joint Research Center Program JPMXP0723833165. He is not employed there. However, our study also thanks them for such an opportunity.
Saga Souhatsu Mathematical Seminar (http://inasa.ms.saga-u.ac.jp/Japanese/saga-souhatsu.html), inviting the author as a speaker, is funded and supported by JST Fusion Oriented REsearch for disruptive Science and Technology JPMJFR202U: the author was a speaker on 2024/7/12 supported by this project.\\
\ \\
{\bf Data availability.} \\
Essentially, data supporting our present study are all here. Note that this respects \cite{kitazawa2} and \cite{kitazawa2} and the present paper both study similar studies on Morse-Bott functions on similar classes of $3$-dimensional closed , connected and orientable manifolds. However, the present paper studies some new problems different from ones in \cite{kitazawa2} and present several new arguments. 

\end{document}